\documentclass[a4paper,reqno,uplatex]{amsart}
\counterwithin{equation}{section}
\usepackage{mathpackages,myThm,pdfcomments}
\begin{document}
\title[Sharp Hardy-Leray inequality for solenoidal fields]{Sharp Hardy-Leray inequality\\ for solenoidal fields }
\author[N. Hamamoto]{Naoki Hamamoto}
\address{\parbox{\linewidth}{
Department of Mathematics, Graduate School of Science, Osaka Metropolitan University,
3-3-138 Sugimoto, Sumiyoshi-ku, Osaka 558-8585, Japan}}
\email{g00glyoe@gmail.com {\rm (N. Hamamoto)}}
\stackMath
\begin{abstract}
We compute the best constant in functional integral inequality called the Hardy-Leray inequalities for solenoidal vector fields on $\mathbb{R}^N$. This gives a solenoidal improvement of the inequalities whose best constants are known for unconstrained fields, and develops of the former work by Costin-Maz'ya \cite{Costin-Mazya} who found the best constant in the Hardy-Leray inequality for axisymmetric solenoidal fields. We derive the same best constant without any symmetry assumption, whose expression can be simplified in relation to the weight exponent.  Moreover, it turns out that the best value cannot be attained in the space of functions satisfying the finiteness of the integrals in the inequality.
\end{abstract}
\maketitle

\section{Introduction}
We study functional inequality called the Hardy-Leray inequality, from the viewpoint that how the best value of the constant changes when we assume that the test functions are vector-valued and solenoidal (namely divergence-free) on $N$-dimensional Euclidean space $\mathbb{R}^N$. 

\subsection{Basic notation and definitions}
\label{subsec:1.1}
Throughout this paper, $N$ denotes an integer and we assume that $N\ge 3$ unless otherwise specified. 
We always use bold letters to denote vectors, e.g., ${\bm x}=(x_1,x_2,\cdots,x_N)\in\mathbb{R}^N$ which is an $N$-dimensional vector. 
 For every open subset $\Omega$ of $\mathbb{R}^N$, the notation ${\bm u}=(u_1,\cdots,u_N)\in C^\infty(\Omega)^N$ means that $\{u_1,\cdots,u_N\}\subset C^\infty(\Omega)$, or equivalently that
\[{\bm u}:\Omega\to\mathbb{R}^N,\qquad {\bm x}\mapsto {\bm u}({\bm x})=(u_1({\bm x}),\cdots,u_N({\bm x}))\]
is a vector field (namely $\mathbb{R}^N$-valued function) with the components 
belonging to the set $C^\infty(\Omega)$ of smooth scalar fields (namely $\mathbb{R}$-valued functions) on $\Omega$. 
The same also applies to other sets of functions: e.g., the notation ${\bm u}\in C_c^\infty(\Omega)^N$ means that ${\bm u}$ is a smooth vector field with compact support on $\Omega$. 

The symbol $\nabla$ denotes the gradient operator which maps every scalar field $f$ on (a domain of) $\mathbb{R}^N$ to the vector field $ \nabla f=\(\frac{\partial f}{\partial x_1},\frac{\partial f}{\partial x_2},\cdots,\frac{\partial f}{\partial x_N}\)$. 
The same also applies to every vector field ${\bm u}$ in the 
componentwise sense:  
we set
\[ \nabla {\bm u}
=\(\nabla u_1,\nabla u_2,\cdots,\nabla u_N\)=\(\frac{\partial u_j}{\partial x_k}\)_{(j,k)\in\{1,2,\cdots,N\}^2}:\Omega\to\mathbb{R}^{N\times N}\]
which is a matrix field. 

The notation ${\bm u}\cdot {\bm v}=\sum_{k=1}^Nu_kv_k$ denotes the standard scalar  product of two vector fields, and $|{\bm u}|=\sqrt{{\bm u}\cdot {\bm u}}$ denotes the modulus of ${\bm u}$. 
The same also applies to matrix fields, 
e.g., \ $\displaystyle \nabla {\bm u}\cdot\nabla\bm{v}=\sum_{j=1}^N\sum_{k=1}^N \frac{\partial u_j}{\partial x_k}\frac{\partial v_j}{\partial x_k}$ \ and  \ $|\nabla\bm{u}|=\sqrt{\nabla\bm{u}\cdot \nabla\bm{u}}.$ 

\subsection{Preceding results and motivation}
\label{subsec:PM}
Hereafter $\gamma$ denotes a real number. 
The classical Hardy-Leray inequality with power weight (or shortly H-L inequality) 
on $\mathbb{R}^N$ is given by
\begin{equation}
 \int_{\mathbb{R}^N}|\nabla {\bm u}|^2|\bm{x}|^{2\gamma}dx\ge \left(\gamma+\mfrac{N-2}{2}\right)^2\int_{\mathbb{R}^N}\frac{|{\bm u}|^2}{|{\bm x}|^2}|\bm{x}|^{2\gamma}dx,
\end{equation}
where  the  constant number $\(\gamma+\frac{N-2}{2}\)^2$ is known to be the sharp for unconstrained vector fields ${\bm u}\in C^\infty(\mathbb{R}^N)^N$ under suitable regularity 
condition. 
This inequality was shown firstly for scalar fields 
by J. Leray \cite{Leray}  for $N=3$ and $\gamma=0$ along his study on the Navier-Stokes equations, as an extension of the $1$-dimensional inequality
by G. H. Hardy\;\cite{Hardy}. The scalar field version of H-L inequality could be trivially extended to vector fields by applying it to each of the components.

One of the questions of interest is whether the best value of the constant in the H-L inequality can exceed the original one 
by imposing on ${\bm u}$ to be {\it solenoidal}, that is,
\[
{\rm div}\:\!\bm{u}=\sum_{k=1}^N\frac{\partial u_k}{\partial x_k}\equiv 0.
\]
Such a problem arises in a natural way in the context of hydrodynamics, as asked by O. Costin and V. G. Maz'ya \cite{Costin-Mazya}, who derived the H-L inequality
 \[\int_{\mathbb{R}^N}|\nabla {\bm u}|^2|\bm{x}|^{2\gamma}dx\ge C_{N,\gamma}\int_{\mathbb{R}^N}\frac{|{\bm u}|^2}{|{\bm x}|^2}|\bm{x}|^{2\gamma}dx\]
for axisymmetric solenoidal fields ${\bm u}$ with the new constant number
\begin{equation}
C_{N,\gamma}= 
\( \gamma + \tfrac{N-2}{2} \)^2 + \min\left\{N-1,\ 2+\min_{\tau\ge0}\(\tau+\tfrac{4(N-1)(\gamma-1)}{\tau+N-1+\(\gamma-\frac{N}{2}\)^2}\)\right\}
\label{CM_orig}
\end{equation}
which we simply call C-M constant.  Here the axisymmetry of a vector field means, in brief, that every component of it along the cylindrical coordinates depends only on the axial distance and the height. Notice that the expression \eqref{CM_orig} 
can be further simplified as
\begin{align}
  C_{N,\gamma}
&=\( \gamma + \tfrac{N-2}{2} \)^2 + \min\left\{N-1,\ 2+\tfrac{4(N-1)(\gamma-1)}{N-1+\(\gamma-\frac{N}{2}\)^2}\right\}
\label{CM_2}
\\&= 
\left\{
\begin{array}{ll}
 \(\gamma+\frac{N-2}{2}\)^2+N-1 &\text{for }\ \gamma\in I_N,
\\[0.5em]
 \(\gamma+\frac{N-2}{2}\)^2 \frac{\(\gamma-\frac{N}{2}\)^2+N+1}{\(\gamma-\frac{N}{2}\)^2+N-1}\ &\text{for }\ \gamma\not\in I_N,
\end{array}
\right.
\label{CM_const}
\end{align}
where $I_N=(\gamma_N^-,\gamma_N^+)\subset\mathbb{R}$ denotes the open interval between 
\[
 \gamma_N^-=\mfrac{N}{2}-\mfrac{N-1}{\sqrt{N+1}+2}\quad\ \text{ and }\quad \gamma_N^+=\left\{
\begin{array}{cl}
 \frac{N}{2}+\frac{N-1}{\sqrt{N+1}-2}&	(N\ge4)     \vspace{0.4em}\\
\infty&(N=3)
\end{array}
\right..
\]
Indeed, one can check by an elementary numerical calculation that the two numbers \eqref{CM_orig} and \eqref{CM_const} coincide, see \cite{HL-const_AdM} for the details.
An advantage of the condition of axisymmetry is that it helps us to easily compute the new best constant, normally without affecting the final result. 
However, it was also pointed out in \cite{RL_div} that the sharpness of the C-M constant under the axisymmetry condition is not always valid when $N\ge 4$; more precisely, it was shown that the best constant for axisymmetric solenoidal fields should be
\[
C^{\rm axis}_{N,\gamma}=
\left\{\begin{array}{cl}
C_{3,\gamma}&(N=3),
\\[0.5em]
\displaystyle \(\gamma+\tfrac{N-2}{2}\)^2+2+\min_{\tau\ge 0}\(\tau+\tfrac{4(N-1)(\gamma-1)}{\tau+N-1+\(\gamma-\frac{N}{2}\)^2}\)
& (N\ge 4).
\end{array}\right.
\]
Combining this fact together with the same numerical calculation in \cite{HL-const_AdM} or Lemma~\ref{lemma:const1}, one can precisely check that the case $C_{N,\gamma}\ne C^{\rm axis}_{N,\gamma}$ occurs when $\gamma\in I_N$ with $N\ge 4$. 
Nevertheless, we may expect that this gap could be reduced by weakening or excluding the condition of axisymmetry, since such a condition seems to play a technical rather than essential role, in order to purely obtain a solenoidal improvement of H-L inequality.

In view of the observation above, there was  an advance in the three-dimensional case: Hamamoto \cite{3D_NA} succeeded in deriving the C-M constant of H-L inequality for solenoidal fields $\bm{u}$ on $\mathbb{R}^3$ without any symmetry assumption at all, after his joint work with Takahashi \cite{swirl_CPAA} investigating the same inequality under a weakened axisymmetry condition. Consequently the axisymmetry assumption for $N=3$ turned out to be completely removable from the solenoidal improvement of H-L inequality. As a matter of course, it is then expected that the same also applies to the higher dimensional case $N\ge 4$ and that the validity of the sharpness of C-M constant could be recovered.  This is the main theme of our present study.

 As a side note, it is also worthy to consider how the best constants change in H-L and R-H inequalities for curl-free vector fields (instead of solenoidal ones).  Some topics related to this problem can be found in 
 the papers \cite{CF_MAAN,CF_JFA,CF_MIA
 }. 

\subsection{Main results}
\label{subsec:1.2}
In order to formally state our main theorem, 
let 
$\mathcal{D}_\gamma(\mathbb{R}^N)
$ denote the space of all smooth vector fields on $\mathbb{R}^N$ satisfying the finiteness of the norm:
\[
\begin{aligned}
&\|\bm{u}\|_{\gamma}:=\(\int_{\mathbb{R}^N}\(|\nabla {\bm u}|^2+\frac{|\bm{u}|^2}{|\bm{x}|^2}\)|{\bm x}|^{2\gamma}dx\)^{1/2}<\infty.
\end{aligned}
\] 
As a matter of fact, one can check that the set of smooth solenoidal fields with compact support on $\mathbb{R}^N\setminus\{\bm{0}\}$  forms a dense subspace of solenoidal fields in $\mathcal{D}_\gamma(\mathbb{R}^N)$  with respect to the norm $\|\cdot\|_{\gamma};$  for the detailed proof, 
see e.g. the discussion in \cite[\S3.1]{RL_CVPD}. 

Now, our first main result reads as follows:
\begin{theorem}
\label{theorem:HL}
Let ${\bm u} \in \mathcal{D}_\gamma(\mathbb{R}^N)$ be a solenoidal field. 
Then the inequality
\begin{equation}
\int_{\mathbb{R}^N} |\nabla {\bm u}|^2 |{\bm x}|^{2\gamma} d{x}\ge C_{N,\gamma} \int_{\mathbb{R}^N} \frac{|{\bm u}|^2}{|{\bm x}|^2} |{\bm x}|^{2\gamma} d{x}
\label{H-L}
\end{equation}
holds with the best constant $C_{N,\gamma}$ given in \eqref{CM_const}, where the equality holds if and only if $\bm{u}\equiv \bm{0}.$
\end{theorem}%
Hence the sharpness of the expression of C-M constant turns out to be valid even for $N\ge 4$ under no assumption of axisymmetry. Moreover,  
one may conclude from the theorem that the strict inequality $C_{N,\gamma}>\(\gamma+\frac{N-2}{2}\)^2$ holds for  all $\gamma\ne-\frac{N-2}{2},$ which says that the solenoidal improvement (almost) always takes effect.

\begin{remark}
By forcibly substituting $N=2$ into \eqref{CM_2} and \eqref{H-L}, we get the two-dimensional H-L inequality
\[
\int_{\mathbb{R}^2}|\nabla\bm{u}|^2|\bm{x}|^{2\gamma}dx\ge C_{2,\gamma}\int_{\mathbb{R}^2}\frac{|\bm{u}|^2}{|\bm{x}|^2}|\bm{x}|^{2\gamma}dx
\]
for solenoidal fields $\bm{u}$ on $\mathbb{R}^2$ with $C_{2,\gamma}=
	\left\{
	\begin{array}{ll}
	\gamma^2+1 &  {\rm for}\ |\gamma-1|<\sqrt{3} 
	\\ \(\frac{(\gamma-1)^2+3}{(\gamma-1)^2+1}\)\gamma^2 &{\rm for}\ |\gamma-1|\ge \sqrt{3} 
	\end{array}
	\right.,$ which reproduces the same result as in \cite{Costin-Mazya} with $N=2.$ Hence, it consequently turns out that Theorem \ref{theorem:HL} is also valid even in the two-dimensional case.
\end{remark}

\subsection{An overview of the rest of this paper}
A key tool that will be used in the proof of Theorem \ref{theorem:HL} is the so-called  
poloidal-toroidal decomposition (or shortly PT) theorem, which was used in the previous work \cite{3D_NA} for only three-dimensional sharp H-L inequality; our recent work \cite{RL_CVPD} also used PT theorem for sharp Rellich-Leray inequality, by quoting the preprint \cite{HL_pre} which corresponds to the present paper. 
The PT theorem in our study, which originates from G. Backus \cite{Backus} on $\mathbb{R}^3$ and 
serves as a specialized version of N. Weck's  \cite{Weck} on $\mathbb{R}^N$,  
 enables us to separate the calculation of the best constant into two computable parts. 
Some techniques on $\mathbb{R}^3$, employed in the previous work, are not allowed in the higher-dimensional case:  we cannot use the concept of ``cross product'' of vectors in general $\mathbb{R}^N$, and there is no way to represent every toroidal field in terms of a single-scalar potential.  To avoid such a difficulty, we derive zero-spherical-mean property of toroidal fields, from which one can readily deduce an appropriate $L^2(\mathbb{S}^{N-1})$ estimate. 


In the remaining content of this paper, we will prove the theorem in the organization as follows: Section~\ref{sec:VC} reviews vector calculus on $\mathbb{R}^N\setminus\{{\bm 0}\}$ in terms of radial-spherical variables. Section~\ref{sec:PT} gives a systematic introduction to the concept of smooth PT fields and  establishes the PT decomposition theorem of solenoidal fields on $\mathbb{R}^N$. 
Section~\ref{sec:main} gives the proof of Theorem~\ref{theorem:HL}, where we compute $C_{N,\gamma}$ to derive the expression \eqref{CM_orig} by making full use of the PT theorem and FET 
and show the non-attainability of $C_{N,\gamma}$ by exploiting the expression \eqref{CM_const}. 

\section{Standard Vector Calculus on {$\dot{\mathbb{R}}^N\cong\mathbb{R}_+\times\mathbb{S}^{N-1}$}}
\label{sec:VC}
We introduce some notations frequently used in this paper and briefly review some basic formulae related to gradient or Laplace operators acting on smooth functions on $\dot{\mathbb{R}}^N$, from the viewpoint of radial-spherical variables.
\subsection{Radial-spherical decomposition of operators }
\label{subsec:RS}
Hereafter we use the abbreviation
\[\dot{\mathbb{R}}^N:=\mathbb{R}^N\setminus\{{\bm 0}\}=\{{\bm x}\in\mathbb{R}^N ;\ {\bm x}\ne {\bm 0}\}.\]
In the sense of differential geometry, this is a smooth manifold diffeomorphic to the product of the half line $\mathbb{R}_+=\left\{{r}\in\mathbb{R}\ ;\ {r}>0\right\}$ and the unit $(N-1)$-sphere $\mathbb{S}^{N-1}=\left\{{\bm x}\in\mathbb{R}^N\ ;\ |{\bm x}|=1\right\}$. In our setting,  we make the identification
\[
\dot{\mathbb{R}}^N\cong \mathbb{R}_+\times\mathbb{S}^{N-1},\qquad{\bm x}\mapsto({r},{\bm \sigma})
\]
by the relation ${\bm x}={r}{\bm \sigma}$ or equivalently 
\begin{equation}
 {r}=|{\bm x}|\quad\text{ and }\quad
 {\bm \sigma}=\dfrac{{\bm x}}{|{\bm x}| }.
  \label{r-sigma}
\end{equation}
In view of the two equations, we may consider $({r},{\bm \sigma})$  as a pair of positive scalar field and unit vector field on $\dot{\mathbb{R}}^N$ which are related by the single equation $\nabla {r}={\bm \sigma}$. 

Every vector field ${\bm u}=(u_1,u_2,\cdots,u_N):\dot{\mathbb{R}}^N\to \mathbb{R}^N$  
has as its components the scalar field $u_R$ and  vector field ${\bm u}_S$  satisfying
\begin{equation}
 {\bm u}={\bm \sigma}u_R+{\bm u}_S\ \quad\text{ and } \quad {\bm \sigma}\cdot {\bm u}_S=0
\qquad{\rm on}\ \, \dot{\mathbb{R}}^N,
\end{equation}
which we call the radial-spherical decomposition  (or shortly RS decomp.) of ${\bm u}$. 
This pair is unique and specified by the equations
\begin{equation}
 u_R={\bm \sigma}\cdot {\bm u}\quad \text{ and }\quad {\bm u}_S={\bm u}-{\bm \sigma}u_R.
\label{uRuS}
\end{equation}

The operator
 \begin{equation}
 \partial_{r}:={\bm \sigma}\cdot\nabla=\sum_{k=1}^N \frac{x_k}{|{\bm x}|}\frac{\partial}{\partial x_k}\quad\text{ resp. }\quad \partial_{r}':=\partial_{r}+\frac{N-1}{{r}}
 \label{der_r} 
 \end{equation}
denotes the radial derivative resp. its skew $L^2$ adjoint, in the sense that
 \[
 \int_{\mathbb{R}^N}f\partial_{r}g\:\!dx=-\int_{\mathbb{R}^N}g\partial_{r}'f\:\! dx
 \]
 holds for all $f,g\in C_c^\infty(\dot{\mathbb{R}}^N)$. 
 When these operators act on vector fields ${\bm u}$, the operation is componentwise; 
it is then easy to check from \eqref{r-sigma} that 
 \begin{equation}
  \partial_{r}{r}=1\quad\text{ and }\quad \partial_{r}{\bm \sigma}= {\bm 0},
\end{equation}
and that the radial derivative commutes with the RS decomp. of vector fields. 


The operator $\nabla_{\!\sigma}$ is defined by 
\[
 \nabla_{\!\sigma}f:=\({r}\nabla f\)_S\quad\ \forall f\in C^\infty(\dot{\mathbb{R}}^N),
\]
which we call the spherical gradient on $\dot{\mathbb{R}}^N$. It is easy to check that 
$\nabla_{\!\sigma}$ commutes with $\partial_{r}$ as well as ${r}$. 
Applying \eqref{uRuS} to ${\bm u}=\nabla f$ yields
\begin{equation}
 \(\nabla f\)_R=\partial_{r}f\quad\text{ and }\quad \frac{1}{{r}}\nabla_{\!\sigma}f=\nabla f-{\bm \sigma}\partial_{r}f
,
\label{nab_s}
\end{equation}
and hence we may simply write as 
\begin{equation}
 \label{nabla}
   \nabla ={\bm \sigma}\partial_{r}+\frac{1}{{r}}\nabla_{\!\sigma}
\end{equation}
which can be understood as a RS decomp. of $\nabla$.

The standard Laplacian $\triangle=\sum_{k=1}^N \frac{\partial^2}{\partial x_k^2}$ is known to have the  expression
\begin{equation} 
\triangle=\partial_{r}'\partial_{r}+\frac{1}{{r}^2}\triangle_\sigma.\label{Lap}
\end{equation}
Here $\triangle_\sigma$ denotes the Laplace-Beltrami operator on $\mathbb{S}^{N-1}$, which we also call the spherical Laplacian. 
As is well known, the operators $\nabla_{\!\sigma}$ and $\triangle_\sigma$ can be expressed explicitly in terms of angular coordinates 
of the spherical polar coordinate system in $\dot{\mathbb{R}}^N$;   
for details, see e.g. \cite[Chapter 2]{Efthimiou-Frye} and also \cite[Section 2]{CF_MAAN}. 

Let us recall that ${\rm div}\:\!{\bm u}=\nabla\cdot {\bm u}=\sum_{k=1}^N \frac{\partial u_k}{\partial x_k}$ is the trace of the matrix field $\nabla {\bm u}$. By extracting its spherical part, we define 
\[\begin{split}
  \nabla_{\!\sigma}\cdot{\bm u}&:={\rm trace}\(\nabla_{\!\sigma}{\bm u}\)={r}\,{\rm trace}\(\nabla {\bm u}-{\bm \sigma}\partial_{r}{\bm u}\)\\&\;={r}\(\nabla\cdot {\bm u}-{\bm \sigma}\cdot\partial_{r}{\bm u}\)
={r}\:\!{\rm div}\:\!{\bm u}-{r}\:\!\partial_{r}u_R.
\end{split}
\]
Here the second equality follows by applying the second equation of \eqref{nab_s}; the operation of $\nabla_{\!\sigma}$ on vector fields is again componentwise (in the sense of \S\ref{subsec:1.1}). 
Then a direct calculation yields 
\[
\nabla_{\!\sigma}\cdot{\bm \sigma}={r}\:\!{\rm div}{\bm \sigma}={\rm div}\(r {\bm \sigma}\)-{\bm \sigma}\cdot\nabla {r}={\rm div}\:\!{\bm x}-1=N-1,\]
which deduces the identity
\begin{equation}
 {\rm div}\:\!{\bm u}
  =\partial_{r}'u_R+\frac{1}{{r}}\nabla_{\!\sigma}\cdot{\bm u}_S
\label{div_RS} 
\end{equation}
as a RS decomp. of the divergence operator. 
By applying the identity to a gradient field $\nabla f={\bm \sigma}\partial_{r}f+\frac{1}{{r}}\nabla_{\!\sigma}f$, 
it directly follows from \eqref{Lap} that
\[\nabla_{\!\sigma}\cdot\nabla_{\!\sigma}f=\triangle_\sigma f 
\qquad\forall f\in C^\infty(\dot{\mathbb{R}}^N)
\]  
as an analogue to the identity $\nabla\cdot\nabla f=\triangle f$. 

For later use, we show the following two lemmas
:
\begin{lemma}[Spherical integration by parts formula]
\label{lemma:div} 
\[\begin{split}
   &\int_{\mathbb{S}^{N-1}}{\bm u}\cdot\nabla_{\!\sigma}f\,\mathrm{d}\sigma=-\int_{\mathbb{S}^{N-1}}(\nabla_{\!\sigma}\cdot {\bm u}_S)f\,\mathrm{d}\sigma
 \end{split}
\]
holds for all ${\bm u}\in C^\infty(\dot{\mathbb{R}}^N)^N$ and $f\in C^\infty(\dot{\mathbb{R}}^N)$, 
where the integrals are taken for any fixed radius. In particular, $\displaystyle\int_{\mathbb{S}^{N-1}}|\nabla_{\!\sigma}f|^2\mathrm{d}\sigma=-\int_{\mathbb{S}^{N-1}}f\triangle_\sigma f\,\mathrm{d}\sigma$.
\end{lemma}
\begin{proof} 
Let $\zeta\in C_c^\infty(\dot{\mathbb{R}}^N)$ be a radially symmetric scalar field. Then integration by parts of \ $({\bm u}\cdot\nabla_{\!\sigma}f)\zeta ={\bm u}_S\cdot\nabla({r}f\zeta )$ \ yields
\[\begin{split}
   \int_{\mathbb{R}^N}({\bm u}\cdot\nabla_{\!\sigma}f)\zeta\:\!dx
   &=-\int_{\mathbb{R}^N}({\rm div}\:\!{\bm u}_S)\:\!{r}f\zeta\:\! dx
   =-\int_{\mathbb{R}^N}(\nabla_{\!\sigma}\cdot{\bm u}_S)f\zeta\:\! dx,
  \end{split}\]
where the last equality follows from \eqref{div_RS}.  
Since the choice of $\zeta$ is arbitrary in the radial coordinate, we get the desired formula,  
with the aid of the fundamental lemma of the calculus of variations. 
\end{proof}
\begin{lemma}
\label{lemma:comm}
\begin{equation}  
\left\{\begin{array}{ll}    \triangle_\sigma({\bm \sigma}f)={\bm \sigma}(\triangle_\sigma-N+1)f +2\:\!\nabla_{\!\sigma}f,    \vspace{0.5em} \\	  \triangle_\sigma\nabla_{\!\sigma}f=\nabla_{\!\sigma}\triangle_\sigma f+(N-3)\nabla_{\!\sigma}f-2\:\!{\bm \sigma}\triangle_\sigma f	\end{array}\right.\label{comm} 
\end{equation}
for all $f\in C^\infty(\dot{\mathbb{R}}^N)$. In particular, if $f$ belongs to an eigenvalue $\alpha$ of $-\triangle_\sigma,$ then
\[
\left\{\begin{array}{ll}    \triangle_\sigma({\bm \sigma}f)=-{\bm \sigma}(\alpha+N-1)f +2\:\!\nabla_{\!\sigma}f,    \vspace{0.5em} \\	  \triangle_\sigma\nabla_{\!\sigma}f=\(-\alpha +N-3\)\nabla_{\!\sigma} f+2\:\!\alpha\:\! {\bm \sigma} f.
\end{array}\right.
\]
\end{lemma}
This fact gives commutation relations between $\triangle_\sigma$ and ${\bm \sigma}$ or $\nabla_{\!\sigma}$. For the proof of the lemma, see e.g. \cite[Lemma 7]{CF_MAAN} or \cite[Lemma 3]{CF_JFA}
\section{Poloidal-toroidal fields}
\label{sec:PT}
A precise definition of poloidal-toroidal fields in general dimension was introduced by Weck \cite{Weck} in the framework of differential forms on $\dot{\mathbb{R}}^N$, including advanced tools such as Hodge dual, interior product, codifferential, etc. However, such a framework requires too wide range of concepts to focus on our current issue. In order to solve this gap, we give a more specialized formalization to introduce the definition of PT fields,  in the framework of minimum required tools in the standard vector calculus. 

\subsection{Pre-poloidal fields and toroidal fields on $\dot{\mathbb{R}}^N$}
\label{subsec:prePT} 
We say that a vector field ${\bm u}=\bm{u}(\bm{x})$ on $\dot{\mathbb{R}}^N$ is {\it pre-poloidal} if there exist two scalar fields $f$ and $g$ satisfying
\begin{equation}
 {\bm u}={\bm x}g+\nabla f\quad\text{ on }\dot{\mathbb{R}}^N.
\label{def_P}
\end{equation}
This definition is equivalent to the existence of  $f,g$ satisfying
\begin{equation}
 {\bm u}={\bm \sigma}g+\nabla_{\!\sigma}f\quad\text{ on }\dot{\mathbb{R}}^N.
\label{Def_P}
\end{equation}
Then the set of all pre-poloidal fields, which we denote by $\mathcal{P}(\dot{\mathbb{R}}^N)$, is a linear space with the invariance property
\begin{equation}
 \left\{\zeta \:\!{\bm u},\  \partial_{r}{\bm u},
\ \triangle_\sigma {\bm u}\right\}\subset\mathcal{P}(\dot{\mathbb{R}}^N)
  \qquad\forall {\bm u}\in\mathcal{P}(\dot{\mathbb{R}}^N),
 \label{inv_P} 
\end{equation}
where $\zeta\in C^\infty(\dot{\mathbb{R}}^N)$ is any radially symmetric scalar field. 
Indeed,  for every ${\bm u}\in\mathcal{P}(\dot{\mathbb{R}}^N)$ the relation $\{\zeta {\bm u},\,\partial_{r}{\bm u}\}\subset\mathcal{P}(\dot{\mathbb{R}}^N)$ easily follows from the second definition formula \eqref{Def_P}, and taking the operation of $\triangle$ on  the first formula \eqref{def_P} readily yields $\triangle {\bm u}\in\mathcal{P}(\dot{\mathbb{R}}^N)$ from the Leibniz rule; 
an easy application of this result also yields $\triangle_\sigma {\bm u}={r}^2\triangle {\bm u}-{r}^2\partial_{r}'\partial_{r}{\bm u}\in\mathcal{P}(\dot{\mathbb{R}}^N)$ from \eqref{Lap}, whence we arrive at \eqref{inv_P}.

A vector field ${\bm u}\in C^\infty(\dot{\mathbb{R}}^N)^N$ is said to be {\it toroidal} if it is spherical and solenoidal
:
\begin{equation}
\begin{array}{rll}
 & {\bm x}\cdot {\bm u}={\rm div}\:\!{\bm u}=0
\vspace{0.25em}
\\
\text{or equivalently} 
&u_R=\nabla_{\!\sigma}\cdot{\bm u}=0
\end{array}
\quad \text{ on }\dot{\mathbb{R}}^N.
\label{def_T}
\end{equation}
We denote by $\mathcal{T}(\dot{\mathbb{R}}^N)$ the set of all toroidal fields. 
Then the same invariance property \eqref{inv_P} also applies to $\mathcal{T}(\dot{\mathbb{R}}^N)$. 
To give an example, let $i<j$ be two integers between $1$ and $N$, and let ${\bm v}^{(i,j)}$ 
be a vector field with the $k$-th component  given by
\begin{equation}
 v^{(i,j)}_k({\bm x})=
  \left\{\begin{array}{cl}
   -x_j   &\text{if }k=i
    \\
	  x_i	&\text{if }k=j
	   \\
	 0&\text{otherwise}	     \end{array}\right.
\label{ex_T}
\end{equation}
for every $k=1,2,\cdots,N$; then it is easy to check that ${\bm v}^{(i,j)}\in\mathcal{T}(\dot{\mathbb{R}}^N)$.

Including the above argument, we summarize some principal properties of the spaces of pre-poloidal fields and toroidal fields: 
\begin{prop}
\label{prop:pre}
All pre-poloidal fields are $L^2(\mathbb{S}^{N-1})$-orthogonal to all toroidal fields, in the sense that
 \begin{equation}
  \int_{\mathbb{S}^{N-1}}{\bm v}\cdot {\bm w}\,\mathrm{d}\sigma=\int_{\mathbb{S}^{N-1}}\nabla \bm{v}\cdot\nabla \bm{w}\,\mathrm{d}\sigma=0
\end{equation}
for all ${\bm v}\in\mathcal{P}(\dot{\mathbb{R}}^N)$ and ${\bm w}\in\mathcal{T}(\dot{\mathbb{R}}^N)$, where the integrals are taken for any radius. Moreover, the two kinds of fields always satisfy
 \[  \left\{\zeta {\bm v},\, \partial_{r}{\bm v},\,\triangle_\sigma {\bm v}\right\}\subset \mathcal{P}(\dot{\mathbb{R}}^N)\quad \text{ and }\quad\left\{\zeta {\bm w},\, \partial_{r}{\bm w},\,\triangle_\sigma {\bm w}\right\}\subset \mathcal{T}(\dot{\mathbb{R}}^N),
\]
where $\zeta\in C^\infty(\dot{\mathbb{R}}^N)$ is any radially symmetric scalar field; namely, the two spaces $\mathcal{P}(\dot{\mathbb{R}}^N)$ and $\mathcal{T}(\dot{\mathbb{R}}^N)$ are invariant under the the multiplication of $\zeta$ and the operations of  $\partial_{r}$ and $\triangle_\sigma$.
\end{prop}
\begin{proof}
It suffices to check the orthogonality. The pre-poloidal property of ${\bm v}$ says that ${\bm v}={\bm \sigma}g+\nabla_{\!\sigma} f$ for some $f$ and $g$, and hence
\[
 {\bm v}\cdot {\bm w}={\bm w}\cdot\nabla_{\!\sigma} f
\]
follows from the spherical property $w_R=0$ of the toroidal field ${\bm w}$.  Then integration by parts of both sides using ${r}\:\!{\rm div}{\bm w}=\nabla_{\!\sigma}\cdot{\bm w}=0$ yields
\[
 \int_{\mathbb{S}^{N-1}}{\bm v}\cdot {\bm w}\:\!\mathrm{d}\sigma=-\int_{\mathbb{S}^{N-1}}(\nabla_{\!\sigma}\cdot{\bm w})\:\!f\:\!\mathrm{d}\sigma
=0.
\] 
This proves the first orthogonality formula. To prove the second, by using \eqref{nabla} and Lemma~\ref{lemma:div},  integration by parts yields
\[
\begin{split}
 \int_{\mathbb{S}^{N-1}}\nabla \bm{v}\cdot\nabla \bm{w}\,\mathrm{d}\sigma
 &=\int_{\mathbb{S}^{N-1}}\Big(\partial_{r}\bm{v}\cdot\partial_{r}\bm{w}+{r}^{-2}\nabla_{\!\sigma} \bm{v}\cdot\nabla_{\!\sigma} \bm{w}\Big)\mathrm{d}\sigma
 \\&=\int_{\mathbb{S}^{N-1}}\partial_{r}{\bm v}\cdot\partial_{r}{\bm w}\,\mathrm{d}\sigma-{r}^{-2}\int_{\mathbb{S}^{N-1}}{\bm v}\cdot(\triangle_\sigma {\bm w})\,\mathrm{d}\sigma=0, 
\end{split} 
\]
where the last equality follows by applying the first orthogonality formula to the fields $\{\partial_{r}{\bm v},{\bm v}\}\subset\mathcal{P}(\dot{\mathbb{R}}^N)$ and $\{\partial_{r}{\bm w},\triangle_\sigma {\bm w}\}\subset\mathcal{T}(\dot{\mathbb{R}}^N)$. 
\end{proof}
As an application of Proposition \ref{prop:pre}, we deduce the following simple but important fact: 
\begin{corollary}[Zero-spherical-mean property of toroidal fields]
\label{corollary}
For every ${\bm u}\in\mathcal{T}(\dot{\mathbb{R}}^N)$, it holds that
\begin{equation}
  \int_{\mathbb{S}^{N-1}}{\bm u}({r}{\bm \sigma})\:\!\mathrm{d}\sigma=0\qquad\forall{r}>0.
\end{equation}
\end{corollary}
\begin{proof}
For every $k\in\{1,2,\cdots,N\}$, let ${\bm e}_k=\nabla x_k$ denote the constant unit vector field parallel to the $x_k$-axis. Then it is clear that ${\bm e}_k\in\mathcal{P}(\dot{\mathbb{R}}^N)$. Therefore, the integration of the $k$-th component $u_k={\bm e}_k\cdot {\bm u}$ of ${\bm u}={\bm u}({r}{\bm \sigma})$ yields 
\[
\int_{\mathbb{S}^{N-1}}u_k\:\!\mathrm{d}\sigma= \int_{\mathbb{S}^{N-1}}{\bm e}_k\cdot {\bm u}\,\mathrm{d}\sigma=0
\]
 by applying Proposition \ref{prop:pre} to the case  $({\bm e}_k,{\bm u})\in\mathcal{P}(\dot{\mathbb{R}}^N)\times\mathcal{T}(\dot{\mathbb{R}}^N)$. Since this equation holds for all $k\in\{1,2,\cdots,N\}$, we arrive at the desired result.
\end{proof}
\subsection{PT  decomposition of solenoidal fields on $\mathbb{R}^N$} 
\label{subsec:PT}
While all toroidal fields are solenoidal, pre-poloidal fields are not necessarily so; we say that a pre-poloidal field is {\it poloidal} whenever it is solenoidal. 

Now let ${\bm u}$ be a solenoidal field smoothly defined on the whole space $\mathbb{R}^N$. 
Notice from the Gauss' divergence theorem that the surface integral of ${\bm u}$ over $\mathbb{S}^{N-1}$ gives 
$
 \int_{\mathbb{S}^{N-1}}{\bm \sigma}\cdot {\bm u}\,\mathrm{d}\sigma=0$  for any radius. 
This equation says that the scalar field $u_R={\bm \sigma}\cdot {\bm u}$ has zero-spherical mean. Then the Poisson-Beltrami equation (equipped with zero-spherical-mean condition)%
\begin{equation}
\left\{\begin{array}{l}
 \triangle_\sigma f=u_R\quad\text{ on }\dot{\mathbb{R}}^N
\\[0.5em]
\int_{\mathbb{S}^{N-1}}f\,\mathrm{d}\sigma=0
\end{array}\right.
\end{equation}
is known to have an unique solution $f,$ which we can express as
\[
f=\triangle_\sigma^{-1}u_R=-\sum_{k=1}^\infty \frac{1}{\alpha_\nu}u_{R,\nu}
,
\]
where $\{u_{R,\nu}\}_{\nu=1}^\infty $ are the spherical harmonics components of $u_R$ defined by the equations
\[
u_R=\sum_{\nu=1}^\infty u_{R,\nu},
\qquad 
-\triangle_\sigma u_{R,\nu}=\alpha_\nu u_{R,\nu} \quad  (\forall \nu\in\mathbb{N}),
\]
see also \cite[Eq. (13)]{Backus} or \cite[Lemma 3]{Weck}. The solution $f$ is called the {\it poloidal potential} of ${\bm u}$.  To understand this naming, let us introduce the second-order derivative operator
\begin{equation}
  {\bm D}:={\bm \sigma}\triangle_\sigma -{r}\partial_{r}'\nabla_{\!\sigma},
\end{equation}
which  we call the {\it poloidal-field generator}.
It maps every scalar field $f$ to a poloidal field
; 
indeed, it is clear that ${\bm D}f\in\mathcal{P}(\dot{\mathbb{R}}^N)$, and that
\[
{\rm div}{\bm D}f=\partial_{r}'\triangle_\sigma f-\partial_{r}'\nabla_{\!\sigma}\cdot\nabla_{\!\sigma}f=0
\]
follows from \eqref{div_RS}
. Moreover, it is easy to check that the vector field
 \[{\bm u}-{\bm D}\triangle_\sigma^{-1}u_R={\bm u}_S+\nabla_{\!\sigma}\triangle_\sigma^{-1}({r}\partial_{r}'u_R)\]
is toroidal  whenever ${\bm u}$ is solenoidal. Hence we have obtained the following fact:
\begin{prop}[PT theorem]
 \label{prop:PT}
 Let ${\bm u}\in C^\infty(\mathbb{R}^N)^N$ be a solenoidal field.  Then there exists an unique pair of poloidal-toroidal fields $({\bm u}_P,{\bm u}_T)\in\mathcal{P}(\dot{\mathbb{R}}^N)\times\mathcal{T}(\dot{\mathbb{R}}^N)$ satisfying
\[{\bm u}={\bm u}_P+{\bm u}_T\qquad{\rm on }\ \dot{\mathbb{R}}^N.\]
 Here the poloidal part has the explicit expression ${\bm u}_P={\bm D}f$ in terms of the poloidal-field generator ${\bm D}$ acting on the poloidal potential $f=\triangle_\sigma^{-1}u_R$.
\end{prop}

\section{Proof of main theorem}
\label{sec:main}
In order to compute the best H-L constant (namely the best constant in the inequality \eqref{H-L})   for solenoidal fields ${\bm u}\in\mathcal{D}_\gamma(\mathbb{R}^N)$, we can assume that ${\bm u}\in C_c^\infty(\dot{\mathbb{R}}^N)^N$ by the density argument, and we will proceed from \S\ref{subsec:red_PT} to \S\ref{subsec:HL_tor} under this assumption. 

\subsection{Reduction to the case of PT fields} 
\label{subsec:red_PT}
Let ${\bm u}\in C_c^\infty(\dot{\mathbb{R}}^N)^N$ be a solenoidal field, and let ${\bm u}={\bm u}_P+{\bm u}_T$ be its PT decomposition  in the sense of  Proposition \ref{prop:PT}.  
If ${\bm u}_P\not\equiv {\bm 0}$ and ${\bm u}_T\not\equiv {\bm 0}$, then the ratio of the two integrals in inequality \eqref{H-L}, which we simply call the  H-L quotient, has the following estimate:
\begin{align}
 \frac{\int_{\mathbb{R}^N}|\nabla {\bm u}|^2|{\bm x}|^{2\gamma}dx}{\int_{\mathbb{R}^N}|{\bm u}|^2|{\bm x}|^{2\gamma-2}dx}
 &=\frac{\int_{\mathbb{R}^N}|\nabla {\bm u}_P|^2|{\bm x}|^{2\gamma}dx+\int_{\mathbb{R}^N}|\nabla {\bm u}_T|^2|{\bm x}|^{2\gamma}dx}{\int_{\mathbb{R}^N}|{\bm u}_P|^2|{\bm x}|^{2\gamma-2}dx+\int_{\mathbb{R}^N}|{\bm u}_T|^2|{\bm x}|^{2\gamma-2}dx}
\\&\ge \min\left\{\frac{\int_{\mathbb{R}^N}|\nabla {\bm u}_P|^2|{\bm x}|^{2\gamma}dx}{\int_{\mathbb{R}^N}|{\bm u}_P|^2|{\bm x}|^{2\gamma-2}dx},\ \frac{\int_{\mathbb{R}^N}|\nabla {\bm u}_T|^2|{\bm x}|^{2\gamma}dx}{\int_{\mathbb{R}^N}|{\bm u}_T|^2|{\bm x}|^{2\gamma-2}dx}\right\}%
\\&\ge\min\left\{C^{\rm pol}_{N,\gamma},\, C^{\rm tor}_{N,\gamma}\right\},
\label{ratio_HL}
\end{align}
where the first equality follows by applying the $L^2(\mathbb{S}^{N-1})$-orthogonality in Proposition \ref{prop:pre}, and where $C^{\rm pol}_{N,\gamma}$ resp. $C^{\rm tor}_{N,\gamma}$ are the best H-L constant for poloidal resp. toroidal fields:
\begin{equation}
\begin{array}{r}\displaystyle
  C^{\rm pol}_{N,\gamma}
=\inf_{{\bm u}\in\mathcal{P}}\frac{\int_{\mathbb{R}^N}|\nabla{\bm u}|^2|{\bm x}|^{2\gamma}dx}{\int_{\mathbb{R}^N}|{\bm u}|^2|{\bm x}|^{2\gamma-2}dx},
\quad\ 
C^{\rm tor}_{N,\gamma}=
\inf_{{\bm u}\in\mathcal{T}}\frac{\int_{\mathbb{R}^N}|\nabla {\bm u}|^2|{\bm x}|^{2\gamma}dx}{\int_{\mathbb{R}^N}|{\bm u}|^2|{\bm x}|^{2\gamma-2}dx}.
\end{array}
\end{equation}
Here the abbreviation ${\bm u}\in\mathcal{P}$ resp. ${\bm u}\in\mathcal{T}$ under the infimum sign means that ${\bm u}$ 
 runs over all poloidal resp. toroidal fields in $C_c^\infty(\dot{\mathbb{R}}^N)^N\setminus\{{\bm 0}\}$. 
By taking the infimum on both sides of \eqref{ratio_HL} over the test solenoidal fields ${\bm u}$, the best H-L constant $C_{N,\gamma}$ for solenoidal fields is found to satisfy
\begin{equation}
 C_{N,\gamma}= \inf_{\substack{{\bm u}\not\equiv {\bm 0},\\ {\rm div}{\bm u}=0}} \frac{\int_{\mathbb{R}^N}|\nabla {\bm u}|^2|{\bm x}|^{2\gamma}dx}{\int_{\mathbb{R}^N}|{\bm u}|^2|{\bm x}|^{2\gamma-2}dx}\ge \min\left\{C^{\rm pol}_{N,\gamma},\,C^{\rm tor}_{N,\gamma}\right\}.
\end{equation}
Since the reverse inequality $C_{N,\gamma}\le\min\{C^{\rm pol}_{N,\gamma},\, C^{\rm tor}_{N,\gamma}\}$  is clear, it then turns out that
\begin{equation}
 C_{N,\gamma}=\min\left\{C^{\rm pol}_{N,\gamma}, \,C^{\rm tor}_{N,\gamma}\right\}.
\label{best_PT}
\end{equation}
Therefore, the computation of $C_{N,\gamma}$ can be decomposed into $C^{\rm pol}_{N,\gamma}$ and $C^{\rm tor}_{N,\gamma}$. 
\subsection{Evaluation of $C^{\rm pol}_{N,\gamma}$}
 Throughout this subsection,  ${\bm u}\in C_c^\infty(\dot{\mathbb{R}}^N)^N\setminus\{\bm{0}\}$ is  assumed to be poloidal. 
Notice from  Proposition \ref{prop:PT} that $\bm{u}$ can be written as
\begin{align}
 {\bm u}&={\bm u}_P={\bm D}g={\bm \sigma}\triangle_\sigma g-{r}\partial_{r}'\nabla_{\!\sigma}g
\\&=\bm{\sigma}\triangle_\sigma g-\(\partial+N-1\)\nabla_{\!\sigma}g
 \quad\text{ on }\dot{\mathbb{R}}^N
\label{PP}
\end{align}
for the poloidal potential $g=\triangle_\sigma^{-1}u_R$, where and hereafter we use the abbreviation
\[
\partial =r\partial_r.
\]
In the traditional way from \cite{Tertikas-Z}, let us express the spherical harmonics decomposition of $\bm{u}$ together with $g$ as
\begin{equation}
\left\{\begin{lgathered}
\bm{u}=\sum_{\nu=1}^\infty \bm{u}_\nu,\quad\ \bm{u}_\nu=\bm{D}g_\nu=-\bm{\sigma}\alpha_\nu g_\nu-(\partial+N-1)\nabla_{\!\sigma}g_\nu,
\\
g=\sum_{\nu=1}^\infty g_\nu,\qquad  -\triangle_\sigma g_\nu=\alpha_\nu g_\nu.
\end{lgathered}\right.
\label{PP_nu}
\end{equation}
Here $\alpha_{(\cdot)}$ denotes the quadratic function given by 
\begin{equation}
\alpha_\nu:=\nu(\nu+N-2)
\label{eigen_Lap}
\end{equation}
which is just the $\nu$-th eigenvalue of $-\triangle_\sigma$ whenever $\nu$ is a nonnegative integer. 
Our goal here is to evaluate the quotient
\[
\displaystyle \frac{\int_{\mathbb{R}^N}|\nabla\bm{u}_\nu|^2|\bm{x}|^{2\gamma}dx}{\int_{\mathbb{R}^N}\frac{|\bm{u}_\nu|^2}{|\bm{x}|^2}|\bm{x}|^{2\gamma}dx}
\]
for each $\nu\in \mathbb{N}$ with $\bm{u}_\nu\not\equiv \bm{0},$ which enables us to find a lower estimate of the same quotient for the whole field $\bm{u}$ since
\begin{equation}
\frac{\int_{\mathbb{R}^N}|\nabla\bm{u}|^2|\bm{x}|^{2\gamma}dx}{\int_{\mathbb{R}^N}\frac{|\bm{u}|^2}{|\bm{x}|^2}|\bm{x}|^{2\gamma}dx}=\frac{\sum_{\nu\in \mathbb{N}}\int_{\mathbb{R}^N}|\nabla\bm{u}_\nu|^2|\bm{x}|^{2\gamma}dx}{\sum_{\nu\in \mathbb{N}}\int_{\mathbb{R}^N}\frac{|\bm{u}_\nu|^2}{|\bm{x}|^2}|\bm{x}|^{2\gamma}dx}\ge \inf_{\nu\in \mathbb{N}}\frac{\int_{\mathbb{R}^N}|\nabla\bm{u}_\nu|^2|\bm{x}|^{2\gamma}dx}{\int_{\mathbb{R}^N}\frac{|\bm{u}_\nu|^2}{|\bm{x}|^2}|\bm{x}|^{2\gamma}dx}.
\qquad
\label{quotient_HL}
\end{equation}
For that purpose, we fix a positive integer $\nu$ for the time being, and we start with taking the transformation of $g_\nu$ into $f$ by the formula
\begin{equation}
 f({\bm x}):=|{\bm x}|^{\gamma+\frac{N-2}{2}} g_\nu({\bm x})
\qquad\forall {\bm x}\in\dot{\mathbb{R}}^N
\qquad\label{BV}
\end{equation}
which stems from Brezis-V{\'a}zquez \cite{Brezis-Vazquez}.
Then $\bm{u}_\nu$ can be expressed in terms of $f$ by the following calculation:
\begin{align}
 r^{\gamma+\frac{N-2}{2}}\bm{u}_\nu&=r^{\gamma+\frac{N-2}{2}}{\bm D}\left({r}^{-\gamma-\frac{N-2}{2}} f\right)
\label{uP_D}
\\&=r^{\gamma+\frac{N-2}{2}}\({\bm \sigma}\triangle_\sigma \big(r^{-\gamma-\frac{N-2}{2}} f\big)-\(\partial+N-1\)\nabla_{\!\sigma}\big(r^{-\gamma-\frac{N-2}{2}} f\big)\)
 \\&=-{\bm \sigma}\alpha_\nu f-\(\partial-\gamma+\tfrac{N}{2}\)\nabla_{\!\sigma}f.
 \label{uP}
\end{align}
Taking the operation of $\partial$ and $\triangle_\sigma$ on this equation yields 
\begin{align}
r^{\gamma+\frac{N-2}{2}}\partial\bm{u}_\nu
&=\(\partial-\gamma-\tfrac{N-2}{2}\)\big(r^{\gamma+\frac{N-2}{2}}\bm{u}_\nu\big)
\\&=\(\partial-\gamma-\tfrac{N-2}{2}\)\(-\bm{\sigma}\alpha_\nu f-\(\partial-\gamma+\tfrac{N}{2}\)\nabla_{\!\sigma}f\)
\\&=-\bm{\sigma}\alpha_\nu \(\partial f-\(\gamma+\tfrac{N-2}{2}\) f\)
\\&\quad -\nabla_{\!\sigma}\(\(\partial-\gamma+\tfrac{N}{2}\)\(\partial f-\(\gamma+\tfrac{N-2}{2}\) f\)\)
\\&=-\bm{\sigma}\alpha_\nu\(\partial f-\(\gamma+\tfrac{N-2}{2}\) f\)
\\&\quad -\nabla_{\!\sigma}\Big(\partial^2f-(2\gamma-1)\partial f+\(\gamma-\tfrac{N}{2}\)\(\gamma+\tfrac{N-2}{2}\)f\Big)
\qquad
\label{du}
\end{align}
and
\begin{align} 
r^{\gamma+\frac{N-2}{2}}\triangle_\sigma \bm{u}_\nu&=-\alpha_\nu\triangle_\sigma\(\bm{\sigma}f\)-\(\partial-\gamma+\tfrac{N}{2}\)\triangle_\sigma\nabla_{\!\sigma}f \\&=-\alpha_\nu\(-{\bm \sigma}(\alpha_\nu+N-1)f +2\:\!\nabla_{\!\sigma}f\)
\\&\quad -\(\partial-\gamma+\tfrac{N}{2}\)\Big(\(-\alpha_\nu +N-3\)\nabla_{\!\sigma} f+2\:\!\alpha_\nu {\bm \sigma} f\Big).
\\&=\bm{\sigma}\alpha_\nu\Big((\alpha_\nu+2\gamma-1)f-2\partial f\Big)
\\&\quad +\nabla_{\!\sigma}\Big(-2\alpha_\nu f+\(\alpha_\nu-N+3\)\(\partial-\gamma+\tfrac{N}{2}\)f\Big)
\label{Ds_u} 
\end{align}
respectively, where the second last equality follows from Lemma~\ref{lemma:comm}. By applying the integration by parts formula
\[
\begin{lgathered}
\int_{\mathbb{S}^{N-1}}\nabla_{\!\sigma}\varphi\cdot\nabla_{\!\sigma} \psi\,\mathrm{d}\sigma=-\int_{\mathbb{S}^{N-1}}\varphi\triangle_\sigma \psi\,\mathrm{d}\sigma,
\end{lgathered}
\]
taking the $L^2$ integral of \eqref{du} with respect to the measure $\frac{1}{r}dr\mathrm{d}\sigma$ yields 
\begin{align} 
&\frac{1}{\alpha_\nu}\int_{\mathbb{R}^N}|\partial_r\bm{u}_\nu|^2|\bm{x}|^{2\gamma}dx
=\frac{1}{\alpha_\nu}\iint_{\mathbb{R}_+\times\mathbb{S}^{N-1}}\left|r^{\gamma+\frac{N-2}{2}}\partial\bm{u}_\nu\right|^2\frac{dr}{r}\:\!\mathrm{d}\sigma 
\\&=\frac{1}{\alpha_\nu}\iint_{\mathbb{R}_+\times\mathbb{S}^{N-1}}\left|\eqref{du}\right|^2\frac{dr}{r}\mathrm{d}\sigma
\\&=\mathop{\iint}_{\mathbb{R}_+\times\mathbb{S}^{N-1}}\( 
\begin{lgathered}
\alpha_\nu\(\partial f-\(\gamma+\tfrac{N-2}{2}\) f\)^2 
\\ +\Big(\partial^2 f+(2\gamma-1)\partial f+\(\gamma-\tfrac{N}{2}\)\(\gamma+\tfrac{N-2}{2}\)f\Big)^2 
\end{lgathered}
\)\frac{dr}{r}\mathrm{d}\sigma\!\!
\\&=\mathop{\iint}_{\mathbb{R}_+\times\mathbb{S}^{N-1}}\(
\begin{lgathered} 
\alpha_\nu\((\partial f)^2
+\(\gamma+\tfrac{N-2}{2}\)^2f^2\)
+\big(\partial^2f\big)^2
\\+(2\gamma-1)^2(\partial f)^2
+\(\gamma-\tfrac{N}{2}\)^2\(\gamma+\tfrac{N-2}{2}\)^2f^2
\\+\(\gamma-\tfrac{N}{2}\)\(\gamma+\tfrac{N-2}{2}\)\(
-2(\partial f)^2\) 
\end{lgathered} 
\)\frac{dr}{r}\mathrm{d}\sigma 
\\&=\mathop{\iint}_{\mathbb{R}_+\times\mathbb{S}^{N-1}}\(
\begin{lgathered} 
\big(\partial^2f\big)^2
+\(\gamma+\tfrac{N-2}{2}\)^2\(\alpha_\nu+\(\gamma-\tfrac{N}{2}\)^2\)f^2
\\ 
+\Big(\alpha_\nu+(2\gamma-1)^2-2\(\gamma+\tfrac{N-2}{2}\)\(\gamma-\tfrac{N}{2}\)\Big)(\partial f)^2
\end{lgathered} 
\)\frac{dr}{r}\mathrm{d}\sigma,\qquad  
\label{L2s_du}
\end{align} 
where the last equality follows by using the integration by parts formula:
\begin{equation}
\int_0^\infty (\partial^j f)(\partial^kf)\frac{dr}{r}=\left\{
\begin{array}{cl}
0&(j+k\ \text{is odd})
\\
(-1)^{\frac{j-k}{2}}\(\partial^{\frac{j+k}{2}}f\)^2&\(j+k\ \text{is even}\)
\end{array}
	\right.,
\label{IbP}
\end{equation}
which holds true since $f$ is assumed to have compact support on $\mathbb{R}^N\setminus\{\bm{0}\}.$
Similarly, taking the $L^2(\mathbb{R}^N)$-scalar product of  \eqref{uP} and \eqref{Ds_u} with respect to the measure $\frac{1}{r}dr\mathrm{d}\sigma$ yields
\begin{align}
&\frac{1}{\alpha_\nu}\int_{\mathbb{R}^N}|\nabla_{\!\sigma}\bm{u}_\nu|^2|\bm{x}|^{2\gamma-2}dx=\frac{1}{\alpha_\nu}\iint_{\mathbb{R}_+\times\mathbb{S}^{N-1}}|\nabla_{\!\sigma}\bm{u}_\nu|^2r^{2\gamma+N-3}dr\mathrm{d}\sigma
\\&=-\frac{1}{\alpha_\nu}\iint_{\mathbb{R}_+\times\mathbb{S}^{N-1}}r^{2\gamma+N-2}\(\bm{u}_\nu\cdot \triangle_\sigma\bm{u}_\nu\)\frac{dr}{r}\mathrm{d}\sigma
\\&=-\frac{1}{\alpha_\nu}\iint_{\mathbb{R}_+\times\mathbb{S}^{N-1}}\eqref{uP}\cdot\eqref{Ds_u}\frac{dr}{r}\mathrm{d}\sigma
\\&=\alpha_\nu\iint_{\mathbb{R}_+\times\mathbb{S}^{N-1}}\Big((\alpha_\nu+2\gamma-1)f^2-2f\partial f\Big)\frac{dr}{r}\mathrm{d}\sigma
\\&\quad +\iint_{\mathbb{R}_+\times\mathbb{S}^{N-1}}\(\partial f-\(\gamma-\tfrac{N}{2}\)f\)\Big((\alpha_\nu-N+3)\(\partial-\gamma+\tfrac{N}{2}\)f-2\alpha _\nu f\Big)\frac{dr}{r}\mathrm{d}\sigma
\\&=\mathop{\iint}_{\mathbb{R}_+\times\mathbb{S}^{N-1}}
\(\begin{lgathered}
(\alpha_\nu -N+3)(\partial f)^2
\\+ \Big(\alpha_\nu (\alpha_\nu +4\gamma-N-1)+(\alpha_\nu -N+3)\(\gamma-\tfrac{N}{2}\)^2\Big)f^2
\end{lgathered}\)\frac{dr}{r}\mathrm{d}\sigma.
\qquad
\label{uDsu}
\end{align}
We are now ready to express in terms of $f$ the integrals in \eqref{H-L}, that is, we have
\begin{align}
\frac{1}{\alpha_\nu}
&\int_{\mathbb{R}^N}|\bm{u}_\nu|^2|{\bm x}|^{2\gamma-2}dx
=\frac{1}{\alpha_\nu}\iint_{\mathbb{R}_+\times\mathbb{S}^{N-1}}|\bm{u}_\nu|^2r^{2\gamma+N-3}dr\mathrm{d}\sigma
\\&=\frac{1}{\alpha_\nu}
\iint_{\mathbb{R}_+\times\mathbb{S}^{N-1}}\left|\eqref{uP}\right|^2\frac{dr}{r}\:\!\mathrm{d}\sigma
\\&=\frac{1}{\alpha_\nu}\iint_{\mathbb{R}_+\times\mathbb{S}^{N-1}}\(\alpha_\nu^2 f^2+\left|\nabla_{\!\sigma}\(\partial f-\(\gamma-\tfrac{N}{2}\)f\)\right|^2\)\frac{dr}{r}\mathrm{d}\sigma
\\&=\iint_{\mathbb{R}_+\times\mathbb{S}^{N-1}}\(\alpha_\nu f^2+\(\partial f-\(\gamma-\tfrac{N}{2}\)f\)^2\)\frac{dr}{r}\mathrm{d}\sigma
\\&=\iint_{\mathbb{R}_+\times\mathbb{S}^{N-1}}\Big(
(\partial f)^2+\(\alpha_\nu+\(\gamma-\tfrac{N}{2}\)^2\)f^2
\Big)\frac{dr}{r}\mathrm{d}\sigma
\qquad\quad \label{L2_u}
\end{align}
and
\begin{align}
&\frac{1}{\alpha_\nu}\int_{\mathbb{R}^N}|\nabla \bm{u}_\nu|^2|{\bm x}|^{2\gamma}dx
 =\frac{1}{\alpha_\nu}\int_{\mathbb{R}^N}\(|\partial_r \bm{u}_\nu|^2+|\bm{x}|^{-2}|\nabla_{\!\sigma}\bm{u}_\nu|^2\)|\bm{x}|^{2\gamma}dx
\\&=\eqref{L2s_du} +\eqref{uDsu}
\\&=\mathop{\iint}_{\mathbb{R}_+\times\mathbb{S}^{N-1}}\(
\begin{lgathered} 
\big(\partial^2f\big)^2
+\(\gamma+\tfrac{N-2}{2}\)^2\(\alpha_\nu+\(\gamma-\tfrac{N}{2}\)^2\)f^2
\\ 
+\Big(\alpha_\nu+(2\gamma-1)^2-2\(\gamma+\tfrac{N-2}{2}\)\(\gamma-\tfrac{N}{2}\)\Big)(\partial f)^2
\\ +(\alpha_\nu -N+3)(\partial f)^2
\\+ \Big(\alpha_\nu (\alpha_\nu +4\gamma-N-1)+(\alpha_\nu -N+3)\(\gamma-\tfrac{N}{2}\)^2\Big)f^2
\end{lgathered}\)\frac{dr}{r}\mathrm{d}\sigma
\\&=\mathop{\iint}_{\mathbb{R}_+\times\mathbb{S}^{N-1}}\(
\begin{lgathered}
\(\partial^2f\)^2+2\(\alpha_\nu+\(\gamma+\tfrac{N-2}{2}\)^2-(N-1)\gamma+1\)(\partial f)^2\\
+\(
\begin{lgathered}
\(\(\gamma+\tfrac{N-2}{2}\)^2+\alpha_\nu-N+3\)\(\alpha_\nu+\(\gamma-\tfrac{N}{2}\)^2\)
\\+4(\gamma-1)\alpha_\nu
\end{lgathered}
\)f^2
\end{lgathered}
\)\frac{dr}{r}\mathrm{d}\sigma.
\qquad \label{L2_du} 
\end{align}
In order to numericalize the derivative operators, 
we now traditionally introduce  the $1$-D Fourier-Emden transform (henceforth, FET for short) $\varphi\mapsto \widehat{\varphi}$ in the radial direction as follows:
\begin{equation}
 \widehat{\varphi}(\tau,{\bm \sigma})=\frac{1}{\sqrt{2\pi}}\int_{\mathbb{R}}e^{-i\tau t}\varphi(e^t {\bm \sigma})dt,  
\quad\text{where}\ i=\sqrt{-1}.
 \label{FET}
\end{equation}
Notice that this transform changes $\partial$ into the algebraic multiplication by $i\tau$:
\begin{equation}
\widehat{\partial\varphi}=(i\tau)^k \widehat{\varphi}.
\label{Fourier_op}
\end{equation}
Also notice from the $L^2(\mathbb{R})$ isometry of FET that
\[
\int_{\mathbb{R}}\left|\widehat{\varphi}(\tau,\bm{\sigma})\right|^2d\tau=\int_{\mathbb{R}}\left|\varphi(e^t\bm{\sigma})\right|^2dt=\int_{\mathbb{R}_+}\left|\varphi(r\bm{\sigma})\right|^2\frac{dr}{r}.
\]
By applying these formulas to $\varphi\in \left\{f,\partial f\right\}$, we get from \eqref{L2_u} that
\begin{align}
\frac{1}{\alpha_\nu}\int_{\mathbb{R}^N}\frac{|\bm{u}_\nu|^2}{|\bm{x}|^2}|\bm{x}|^{2\gamma}dx&=\iint_{\mathbb{R}\times\mathbb{S}^{N-1}}\Big(\big|\widehat{\partial f}\big|^2+\(\alpha_\nu+\(\gamma-\tfrac{N}{2}\)^2\)|\widehat{f}|^2\Big)d\tau\mathrm{d}\sigma
\\&=\iint_{\mathbb{R}\times\mathbb{S}^{N-1}}\(\tau^2+\alpha_\nu+\(\gamma-\tfrac{N}{2}\)^2\)|\widehat{f}|^2d\tau\mathrm{d}\sigma
\\&=\iint_{\mathbb{R}\times\mathbb{S}^{N-1}}Q_0\(\tau^2,\alpha_\nu\)|\widehat{f}|^2d\tau\mathrm{d}\sigma,
\label{L2_Q0}
\end{align}
where $Q_0(\cdot,\cdot)$ denotes the polynomial given by
\begin{equation}
Q_0\(\tau,{a}\):=\tau+{a}+\(\gamma-\tfrac{N}{2}\)^2.
\label{poly:Q0}
\end{equation}
Similarly, by applying FET to $\varphi\in \{f,\partial f,\partial^2f\}$ we get from \eqref{L2_du} that
\begin{equation}
\frac{1}{\alpha_\nu}\int_{\mathbb{R}^N}|\nabla\bm{u}_\nu|^2|\bm{x}|^{2\gamma}dx=\iint_{\mathbb{R}\times\mathbb{S}^{N-1}}Q_1\(\tau^2,\alpha_\nu\)|\widehat{f}|^2d\tau\mathrm{d}\sigma,
\qquad
\label{L2_Q1}
\end{equation}
where 
\begin{align}
Q_1(\tau,{a})&:=\tau^2+2\(\alpha_\nu+\(\gamma+\tfrac{N-2}{2}\)^2-(N-1)\gamma+1\)\tau
\\ &\quad\ +
\begin{lgathered}
\(\(\gamma+\tfrac{N-2}{2}\)^2+\alpha_\nu-N+3\)\(\alpha_\nu+\(\gamma-\tfrac{N}{2}\)^2\)
+4(\gamma-1)\alpha_\nu
\end{lgathered}
\\&=\(\tau+{a}-N+3+\(\gamma +\tfrac{N-2}{2}\)^2\)Q_0(\tau,{a})+4(\gamma-1){a}.
\label{poly:Q1}
\end{align}
We are now in the position to begin to evaluate $C^{\rm pol}_{N,\gamma}:$ taking the ratio of \eqref{L2_Q1} to \eqref{L2_Q0} yields
\begin{align}
\frac{\int_{\mathbb{R}^N}|\nabla\bm{u}_\nu|^2|\bm{x}|^{2\gamma}dx}{\int_{\mathbb{R}^N}\frac{|\bm{u}_\nu|^2}{|\bm{x}|^2}|\bm{x}|^{2\gamma}dx}&=\frac{\iint_{\mathbb{R}\times\mathbb{S}^{N-1}}Q_1\(\tau^2,\alpha_\nu\)|\widehat{f}|^2d\tau\mathrm{d}\sigma}{\iint_{\mathbb{R}\times\mathbb{S}^{N-1}}Q_0\(\tau^2,\alpha_\nu\)|\widehat{f}|^2d\tau\mathrm{d}\sigma}
\\&\ge \inf_{\tau\in \mathbb{R}}\frac{Q_1(\tau^2,\alpha_\nu)}{Q_0(\tau^2,\alpha_\nu)}
=\inf_{\tau\ge 0}\frac{Q_1(\tau,\alpha_\nu)}{Q_0(\tau,\alpha_\nu)}.
\end{align}
Combine this result with \eqref{quotient_HL}, then we get
\begin{align}
\frac{\int_{\mathbb{R}^N}|\nabla\bm{u}|^2|\bm{x}|^{2\gamma}dx}{\int_{\mathbb{R}^N}\frac{|\bm{u}|^2}{|\bm{x}|^2}|\bm{x}|^{2\gamma}dx}&\ge \inf_{\tau\ge 0}\inf_{\nu\in \mathbb{N}}\frac{Q_1(\tau,\alpha_\nu)}{Q_0(\tau,\alpha_\nu)}.
\label{infinf}
\end{align}
 In order to evaluate the right-hand side, we see from \eqref{poly:Q0} and \eqref{poly:Q1} that
 \begin{equation}
  \frac{Q_1(\tau,{a})}{Q_0(\tau,{a})}=\tau+{a}-N+3+\(\tau+\tfrac{N-2}{2}\)^2+\frac{4(\gamma-1){a}}{\tau+{a}+\(\gamma-\frac{N}{2}\)^2},
\label{ratio:Q1/Q0}
 \end{equation}
 and that it is monotone increasing in ${a}\ge\alpha_1$ for each $\tau\ge0$; indeed, a straightforward calculation yields
\[\begin{split}
 \frac{\partial}{\partial{a}}\frac{Q_1(\tau,{a})}{Q_0(\tau,{a})}
&=1+\frac{\partial}{\partial{a}}\frac{4(\gamma-1){a}}{\tau+{a}+\(\gamma-\frac{N}{2}\)^2}
 =1+\frac{4(\gamma-1)\(\tau+\(\gamma-\frac{N}{2}\)^2\)}{\(\tau+{a}+\(\gamma-\frac{N}{2}\)^2\)^2}
 \\&=\frac{\(\begin{array}{l}
	   \tau^2+2\(\(\gamma-\frac{N-2}{2}\)^2+{a}+N-3\)\tau\vspace{0.25em}
\\ +\(\gamma-\frac{N}{2}\)^2\(\(2+\gamma-\frac{N}{2}\)^2+2({a}+N-4)\)+{a}^2\end{array}
 \)}{\(\tau+{a}+\(\gamma-\frac{N}{2}\)^2\)^2}
>0
\end{split}
\]
from ${a}+N\ge\alpha_1+N=2N-1>4$, whence
\[
\inf_{\nu\in \mathbb{N}}\frac{Q_1(\tau,\alpha_\nu)}{Q_0(\tau,\alpha_\nu)}=\frac{Q_1(\tau,\alpha_1)}{Q_0(\tau,\alpha_1)}.
\]
Therefore, returning to \eqref{infinf} 
we get
\begin{equation}
\begin{aligned}
 \frac{\int_{\mathbb{R}^N}|\nabla{\bm u}|^2|\bm{x}|^{2\gamma}dx}{\int_{\mathbb{R}^N}\frac{|{\bm u}|^2}{|\bm{x}|^2}|\bm{x}|^{2\gamma}dx}
&\ge \min_{\tau\ge0}\frac{Q_1(\tau,\alpha_1)}{Q_0(\tau,\alpha_1)}.
\end{aligned}
\end{equation}
In other words, it turns out from \eqref{ratio:Q1/Q0} that the inequality
\begin{equation}
 \int_{\mathbb{R}^N}|\nabla {\bm u}|^2|{\bm x}|^{2\gamma}dx\ge C^{\rm pol}_{N,\gamma}\int_{\mathbb{R}^N}\frac{|{\bm u}|^2}{|\bm{x}|^2}|{\bm x}|^{2\gamma}dx
\label{HL_P}
\end{equation}
holds for all poloidal fields ${\bm u}\in C_c^\infty(\dot{\mathbb{R}}^N)^N$ with the constant number
\begin{align}
 C^{\rm pol}_{N,\gamma}&=\min_{\tau\ge0}\frac{Q_1(\tau,\alpha_1)}{Q_0(\tau,\alpha_1)} 
 \\&
 =\left(\gamma+\frac{N-2}{2}\right)^2+2+\min_{\tau\ge0}\left(\tau+\frac{4(N-1)(\gamma-1)}{\tau+\big(\gamma-\frac{N}{2}\big)^2+N-1}\right).\qquad 
\label{C_Pg}
\end{align}

\subsection{Optimality of $C^{\rm pol}_{N,\gamma}$ and a correction term}
\label{subsec:optimal_pol}
Here we consider only the case when 
the minimum in \eqref{C_Pg} is attained by $\tau=0$:
\begin{equation}
\min_{\tau\ge0}\left(\tau+\frac{4(N-1)(\gamma-1)}{\tau+\big(\gamma-\frac{N}{2}\big)^2+N-1}\right)=\frac{4(N-1)(\gamma-1)}{\big(\gamma-\frac{N}{2}\big)^2+N-1},
\label{min_Q1/Q0}
\end{equation}
which is equivalent to  \[C^{\rm pol}_{N,\gamma}=\frac{Q_1(0,\alpha_1)}{Q_0(0,\alpha_1)} =\left(\gamma+\tfrac{N-2}{2}\right)^2 \mfrac{\big(\gamma-\frac{N}{2}\big)^2+N+1}{\big(\gamma-\frac{N}{2}\big)^2+N-1}.\] 
Under this assumption, let us show the sharpness of the constant $C^{\rm pol}_{N,\gamma}$ in the inequality \eqref{HL_P}.  To do so, 
define
$\{{h}_n\}_{n\in\mathbb{N}}\subset C_c^\infty(\dot{\mathbb{R}}^N)$ as a sequence of scalar fields by
\begin{equation}
{h}_n({\bm x})=\frac{x_1}{|\bm{x}|}\zeta\left(\tfrac{1}{n}\log|\bm{x}|\right)
\qquad  \forall n\in\mathbb{N},\ \ \forall {\bm x}\in\dot{\mathbb{R}}^N,
\label{seq_hn}
\end{equation}
where $\zeta:\mathbb{R}\to\mathbb{R}$ is a smooth function $\not\equiv0$ with compact support on $\mathbb{R}$\:\!; notice that the eigenequation
\[
 -\triangle_\sigma {h}_n=\alpha_1 {h}_n\quad\text{ on }\dot{\mathbb{R}}^N\quad(\forall n\in\mathbb{N})
\]
holds from $-\triangle_\sigma x_1=\alpha_1 x_1$. In this setting, apply \eqref{BV}, \eqref{L2_Q0} and \eqref{L2_Q1} to the case
\[
\left\{\nu=1,\quad\bm{u}(\bm{x})=\bm{u}_1(\bm{x})=\bm{D}\(|\bm{x}|^{-\gamma-\frac{N-2}{2}}{h}_n(\bm{x})\),\quad f(\bm{x})={h}_n(\bm{x})\right\},
\]
then we have
\[
\begin{split}
 \frac{\int_{\mathbb{R}^N}|\nabla {\bm u}|^2|\bm{x}|^{2\gamma}dx}{\int_{\mathbb{R}^N}|{\bm u}|^2|\bm{x}|^{2\gamma-2}dx}
&=\frac{\iint_{\mathbb{R}\times\mathbb{S}^{N-1}} Q_1(\tau^2,\alpha_1)|\widehat{h_n}|^2d\tau\mathrm{d}\sigma}{\iint_{\mathbb{R}\times\mathbb{S}^{N-1}} Q_0(\tau^2,\alpha_1)|\widehat{{h}_n}|^2d\tau\mathrm{d}\sigma}
 =\frac{\int_{\mathbb{R}}Q_1\(\tau^2,\alpha_1\)|\widehat{\zeta}\(n\tau\)|^2d\tau}{\int_{\mathbb{R}}Q_0\(\tau^2,\alpha_1\)|\widehat{\zeta}\(n\tau\)|^2d\tau}
 \\&=\frac{\int_{\mathbb{R}}Q_1\((\tau/n)^2,\alpha_1\)|\widehat{\zeta}(\tau)|^2d\tau}{\int_{\mathbb{R}}Q_0\((\tau/n)^2,\alpha_1\)|\widehat{\zeta}(\tau)|^2d\tau}
 =\frac{Q_1(0,\alpha_1)\int_{\mathbb{R}}|\zeta({t})|^2d{t}+O(n^{-2})}{Q_0(0,\alpha_1)\int_{\mathbb{R}}|\zeta({t})|^2d{t}+O(n^{-2})}
\\& \longrightarrow\ \frac{Q_1(0,\alpha_1)}{Q_0(0,\alpha_1)}\quad (n\to \infty ).
\end{split}
\]
This proves the desired sharpness of $C^{\rm pol}_{N,\gamma}$. 

Next, we derive a correction term to make the inequality \eqref{HL_P} stronger. 
To this end, returning to \eqref{ratio:Q1/Q0}, we see that
\begin{align}
 & \frac{1}{\tau}\(\frac{Q_1(\tau,\alpha_1)}{Q_0(\tau,\alpha_1)}-\frac{Q_1(0,\alpha_1)}{Q_0(0,\alpha_1)}\)
 \\&=1+\frac{4(\gamma-1)(N-1)}{\tau}\(\frac{1}{\tau+\big(\gamma-\frac{N}{2}\big)^2+N-1}-\frac{1}{\big(\gamma-\frac{N}{2}\big)^2+N-1}\)
 \\&=1-\frac{4(\gamma-1)(N-1)}{\(\tau+\big(\gamma-\frac{N}{2}\big)^2+N-1\)\(\big(\gamma-\frac{N}{2}\big)^2+N-1\)}=:{G}(\tau)
\label{c0}
\end{align}
is a monotone function in $\tau>0$, whence
\[
 \inf_{\tau>0}{G} (\tau)=\min\left\{{G}(\infty),{G}(0)\right\}=\min\left\{1,{G}(0)\right\}.
\]
By abbreviating as
\begin{equation}
c_{N,\gamma}=\min\left\{1,{G}(0)\right\},
\label{c1}
\end{equation}
we then see from \eqref{L2_Q0} and \eqref{L2_Q1} that
 \begin{align}
  \int_{\mathbb{R}^N}&|\nabla {\bm u}_\nu|^2|{\bm x}|^{2\gamma}dx-C^{\rm pol}_{N,\gamma}\int_{\mathbb{R}^N}|{\bm u}_\nu|^2|{\bm x}|^{2\gamma-2}dx
\\&=\alpha_\nu\iint_{\mathbb{R}\times\mathbb{S}^{N-1}} \(Q_1(\tau^2,\alpha_\nu)-\frac{Q_1(0,\alpha_1)}{Q_0(0,\alpha_1)}Q_0(\tau^2,\alpha_\nu)\)|\widehat{f}|^2d\tau\:\!\mathrm{d}\sigma
\\&\ge\alpha_\nu c_{N,\gamma}\iint_{\mathbb{R}\times\mathbb{S}^{N-1}} \tau^2 Q_0(\tau^2,\alpha_\nu)|\widehat{f}|^2d\tau\:\!\mathrm{d}\sigma
\\&=\alpha_\nu c_{N,\gamma}\iint_{\mathbb{R}\times\mathbb{S}^{N-1}}\(\big|\tau^2\widehat{f}\big|^2+\(\alpha_\nu+\(\gamma-\tfrac{N}{2}\)^2\)\big|\tau \widehat{f}\big|^2\)d\tau\mathrm{d}\sigma
\\&=\alpha_\nu c_{N,\gamma}\iint_{\mathbb{R}_+\times\mathbb{S}^{N-1}}\(\(\partial^2f\)^2+\(\alpha_\nu+\(\gamma-\tfrac{N}{2}\)^2\)\(\partial f\)^2\)\frac{dr}{r}\mathrm{d}\sigma
\\&=c_{N,\gamma}\iint_{\mathbb{R}_+\times\mathbb{S}^{N-1}}\left|\partial \eqref{uP}\right|^2\frac{dr}{r}\mathrm{d}\sigma
\\&=c_{N,\gamma}\iint_{\mathbb{R}_+\times\mathbb{S}^{N-1}}\left|\partial \(r^{\gamma+\frac{N-2}{2}}{\bm u}_\nu\)\right|^2\frac{dr}{r}\mathrm{d}\sigma.
\\&=c_{N,\gamma}\int_{\mathbb{R}^N}\left|\bm{x}\cdot\nabla \(r^{\gamma+\frac{N-2}{2}}{\bm u}_\nu\)\right|^2|\bm{x}|^{-N}dx.
\end{align}
Therefore, by taking the summation over $\nu\in \mathbb{N},$ we eventually get the inequality
\begin{align}
  \int_{\mathbb{R}^N}|\nabla {\bm u}|^2|{\bm x}|^{2\gamma}dx\ge C^{\rm pol}_{N,\gamma}\int_{\mathbb{R}^N}|{\bm u}|^2|{\bm x}|^{2\gamma-2}dx
 + c_{N,\gamma}\mathcal{R}_{N,\gamma}[{\bm u}] 
\label{HL_rem_P}
\end{align}
for all poloidal fields ${\bm u}\in C_c^\infty(\dot{\mathbb{R}}^N)^N$, where $\mathcal{R}_{N,\gamma}[\,\cdot\,]$ denotes the nonnegative functional given by
\begin{equation}
 \mathcal{R}_{N,\gamma}[{\bm u}]:= \int_{\mathbb{R}^N}\left|{\bm x}\cdot\nabla\(|{\bm x}|^{\gamma+\frac{N-2}{2}}{\bm u}\)\right|^2|{\bm x}|^{-N}dx.\label{rem}
\end{equation}
In particular,  as long as $c_{N,\gamma}>0$, 
the inequality \eqref{HL_rem_P} is stronger than \eqref{HL_P}.

As a side note, the optimality of $C^{\rm pol}_{N,\gamma}$ in \eqref{C_Pg} holds even when \eqref{min_Q1/Q0} is not satisfied, though the proof becomes complicated (see \cite[\S4.3]{HL_pre}) and we do not know whether the difference between both sides of \eqref{HL_P} can be bounded  below by a norm of ${\bm u}$. For our present purpose, however, this case is fortunately avoidable . 



\subsection{Evaluation and optimality 
of $C^{\rm tor}_{N,\gamma}$}
\label{subsec:HL_tor}
Let $\bm{u}$ be a toroidal field in $C_c^\infty(\dot{\mathbb{R}}^N)^N$ and let $\bm{v}$ be its Brezis-V\'{a}zquez type transform given by
\[
{\bm v}({\bm x})=|{\bm x}|^{\gamma+\frac{N-2}{2}}{\bm u}({\bm x})
\]
which is also a toroidal field in $C_c^\infty (\dot{\mathbb{R}}^N)^N.$ Then we have
\begin{align}
\int_{\mathbb{R}^N}|\partial\bm{u}&|^2|\bm{x}|^{2\gamma-2}dx
=\iint_{\mathbb{R}_+\times\mathbb{S}^{N-1}}\left|\partial\(r^{-\gamma-\frac{N-2}{2}}\bm{v}\)\right|^2r^{2\gamma+N-3}dr\mathrm{d}\sigma
\\&=\iint_{\mathbb{R}_+\times\mathbb{S}^{N-1}}\left|-\(\gamma+\tfrac{N-2}{2}\)\bm{v}+\partial\bm{v}\right|^2\frac{dr}{r}\mathrm{d}\sigma
\\&=\iint_{\mathbb{R}_+\times\mathbb{S}^{N-1}}\(\(\gamma+\tfrac{N-2}{2}\)^2|\bm{v}|^2+\left|\partial\bm{v}\right|^2\)\frac{dr}{r}\mathrm{d}\sigma
\\&=\(\gamma+\tfrac{N-2}{2}\)^2\int_{\mathbb{R}^N}|\bm{u}|^2|\bm{x}|^{2\gamma-2}dx+\mathcal{R}_{N,\gamma}[\bm{u}],
\label{L2du_tor}
\end{align}
where the second last equality follows by integration by parts together with the compactness of the support of $\bm{v}$ and 
where $\mathcal{R}_{N,\gamma}[\,\cdot\,]$ is the same functional as in \eqref{rem}. 
On the other hand, recall from Corollary \ref{corollary} that 
every toroidal field has zero-spherical mean; then, by considering the spherical harmonics expansion of 
${\bm u}$, 
we easily get the $L^2(\mathbb{S}^{N-1})$ estimate  
\begin{align}
 &\int_{\mathbb{S}^{N-1}}|\nabla_{\!\sigma}{\bm u}|^2\mathrm{d}\sigma\ge\alpha_1\int_{\mathbb{S}^{N-1}}|{\bm u}|^2\mathrm{d}\sigma\quad\ (\text{for each radius}),
 \\\text{whence}\quad&
 \int_{\mathbb{R}^N}|\nabla_{\!\sigma}{\bm u}|^2|\bm{x}|^{2\gamma-2}dx\ge\alpha_1\int_{\mathbb{R}^N}|{\bm u}|^2|{\bm x}|^{2\gamma-2}dx.
 \label{L2Dsu_tor}
\end{align}
Therefore, by way of the identity $|\nabla\bm{u}|^2=\(|\partial\bm{u}|^2+|\nabla_{\!\sigma}\bm{u}|^2\)|\bm{x}|^{-2},$ taking both sides of \eqref{L2du_tor} plus \eqref{L2Dsu_tor} yields
\begin{equation}
 \int_{\mathbb{R}^N}|\nabla {\bm u}|^2|{\bm x}|^{2\gamma}dx
  \ge C^{\rm tor}_{N,\gamma}\int_{\mathbb{R}^N}|{\bm u}|^2|{\bm x}|^{2\gamma-2}dx+\mathcal{R}_{N,\gamma}[{\bm u}]
\label{H-L_tor_rem}  
\end{equation}
with the constant number
\begin{equation}
 C^{\rm tor}_{N,\gamma}=\left(\gamma+\tfrac{N-2}{2}\right)^2+N-1.
\label{C_Tg}  
\end{equation}
Notice that the equality in \eqref{H-L_tor_rem} holds whenever $-\triangle_\sigma {\bm u}=\alpha_1 {\bm u}$ on $\dot{\mathbb{R}}^N$. 
In particular, we have obtained the H-L inequality
\begin{equation}
 \int_{\mathbb{R}^N}|\nabla {\bm u}|^2|{\bm x}|^{2\gamma}dx
  \ge C^{\rm tor}_{N,\gamma}\int_{\mathbb{R}^N}|{\bm u}|^2|{\bm x}|^{2\gamma-2}dx
\label{HL_T}
\end{equation}
for toroidal fields ${\bm u}\in C_c^\infty(\dot{\mathbb{R}}^N)^N$. 
To show that the constant $C^{\rm tor}_{N,\gamma}$ is sharp in this inequality, we put 
\[
 {\bm v}^{(1,2)}({\bm x})=(-x_2,x_1,0,\cdots,0)\qquad\forall {\bm x}\in\dot{\mathbb{R}}^N,
\]
which is the same as in \eqref{ex_T} and satisfies  
the eigenequation
\begin{equation}
 -\triangle_\sigma {\bm v}^{(1,2)}=\alpha_1 {\bm v}^{(1,2)}\quad\text{ on } \dot{\mathbb{R}}^N.
\end{equation}
In this setting, for every $n\in \mathbb{N}$ define $\bm{v}$ and $\bm{u}$ by
\[
\left.\begin{array}{l}
 {\bm v}({\bm x})=\zeta\left(\mfrac{\log|{\bm x}|}{n}\right){\bm v}^{(1,2)}({\bm x}/|{\bm x}|)
 \vspace{0.5em} \\
{\bm u}({\bm x})= |{\bm x}|^{-\gamma-\frac{N-2}{2}}{\bm v}({\bm x})\end{array}
\right\}\quad \forall {\bm x}\in\dot{\mathbb{R}}^N,\  \forall n\in\mathbb{N} ,
\]
where  $\zeta:\mathbb{R}\to\mathbb{R}$ is a smooth function $\not\equiv0$ with compact support on $\mathbb{R}$.
Then, since the equality condition of \eqref{H-L_tor_rem} is satisfied, we get
\[
  \int_{\mathbb{R}^N}|\nabla {\bm u}|^2|{\bm x}|^{2\gamma}dx
  = C^{\rm tor}_{N,\gamma}\int_{\mathbb{R}^N}|{\bm u}|^2|{\bm x}|^{2\gamma-2}dx+\mathcal{R}_{N,\gamma}[{\bm u}]\quad\ (\forall n\in\mathbb{N}).
\] 
Dividing both sides by $\int_{\mathbb{R}^N}|{\bm u}|^2|{\bm x}|^{2\gamma-2}dx=\int_{\mathbb{R}^N}|{\bm v}|^2|\bm{x}|^{-N}dx$
,  we obtain
\[
\begin{split}
& \frac{\int_{\mathbb{R}^N}|\nabla {\bm u}|^2|{\bm x}|^{2\gamma}dx}{\int_{\mathbb{R}^N}|{\bm u}|^2|{\bm x}|^{2\gamma-2}dx}-C^{\rm tor}_{N,\gamma}
\\& =\frac{\mathcal{R}_{N,\gamma}[{\bm u}]}{\int_{\mathbb{R}^N}|{\bm v}|^2|\bm{x}|^{-N}dx} 
=\frac{\int_{\mathbb{R}^N}|\partial_{}{\bm v}|^2|\bm{x}|^{-N}dx}{\int_{\mathbb{R}^N}|{\bm v}|^2|\bm{x}|^{-N}dx}
 =\frac{1}{n^2}\frac{\int_{\mathbb{R}}(\zeta'({s}))^2d{s}}{\int_{\mathbb{R}}(\zeta({s}))^2d{s}}
\ \underset{(n\to\infty)}{\longrightarrow}
 0,
\end{split} 
\]
which proves the desired sharpness of $C^{\rm tor}_{N,\gamma}$.

\subsection{
Non-attainability of $C_{N,\gamma}$} 
\label{subsec:concl}
Substituting \eqref{C_Pg} and \eqref{C_Tg} into \eqref{best_PT}, we see 
that the inequality \eqref{H-L} holds for all solenoidal fields ${\bm u}$ in $C_c^\infty(\dot{\mathbb{R}}^N)^N$ and hence in $\mathcal{D}_\gamma(\dot{\mathbb{R}}^N)$ with the best constant number
\[
\begin{split}
  C_{N,\gamma}&=\min\left\{C^{\rm pol}_{N,\gamma},C^{\rm tor}_{N,\gamma}\right\}
 \\&=\left(\gamma+\mfrac{N-2}{2}\right)^2+\min \left\{2+\min_{\tau\ge0}\left(\tau+\mfrac{4(N-1)(\gamma-1)}{\tau+\big(\gamma-\frac{N}{2}\big)^2+N-1}\right),\ N-1\right\}
\\&= 
\left\{
\begin{array}{ll}
 \(\gamma+\frac{N-2}{2}\)^2+N-1 &\text{for }\ \gamma\in I_N
\\[0.5em]
 \(\gamma+\frac{N-2}{2}\)^2 \frac{\(\gamma-\frac{N}{2}\)^2+N+1}{\(\gamma-\frac{N}{2}\)^2+N-1}\ &\text{for }\ \gamma\not\in I_N
\end{array}
\right..
\end{split}
\]
Here the last equality is merely a repeat of \eqref{CM_const}; 
more precisely, for our present purpose, we shall exploit the following two facts:
\begin{lemma}[\cite{HL-const_AdM}]
\label{lemma:const1}
It holds that
 \[\left\{\begin{array}{lll}
 C^{\rm pol}_{N,\gamma}>C^{\rm tor}_{N,\gamma}
 &\text{for }
\gamma\in I_N
\vspace{0.5em}
\\ 
 C^{\rm tor}_{N,\gamma}\ge C^{\rm pol}_{N,\gamma}
=\(\gamma+\frac{N-2}{2}\)^2 \frac{\(\gamma-\frac{N}{2}\)^2+N+1}{\(\gamma-\frac{N}{2}\)^2+N-1}&
\text{for }\gamma\not\in I_N
\end{array}
\right..
\]
\end{lemma}
%

\begin{lemma}[\cite{HL-const_AdM}]
\label{lemma:const2}
 Let $B$ be the quartic function given by
\[
 B(\lambda)=\lambda^4+(N-1)\(2\lambda^2-4\lambda-N+3\)\quad\ \forall\lambda\in\mathbb{R}.
\]
If $\gamma$ satisfies $B\(\gamma-\frac{N}{2}\)\le0$, then the inequality $C^{\rm tor}_{N,\gamma}+2\le C^{\rm pol}_{N,\gamma}$ holds true.
\end{lemma}

Along these lemmas, let us show the non-attainability of $C_{N,\gamma}$ 
separately in the two cases:
\subsubsection*{The case $\gamma\in I_N$
}
For every solenoidal field ${\bm u}\in C_c^\infty(\dot{\mathbb{R}}^N)^N$, the calculation of \eqref{HL_P}$_{{\bm u}={\bm u}_P}$ plus \eqref{H-L_tor_rem}$_{{\bm u}={\bm u}_T}$  yields
\[
\begin{split}
  \int_{\mathbb{R}^N}&|\nabla {\bm u}|^2|{\bm x}|^{2\gamma}dx
=  \int_{\mathbb{R}^N}|\nabla {\bm u}_P|^2|{\bm x}|^{2\gamma}dx+  \int_{\mathbb{R}^N}|\nabla {\bm u}_T|^2|{\bm x}|^{2\gamma}dx
\\& 
 \ge C^{\rm pol}_{N,\gamma}\int_{\mathbb{R}^N}|{\bm u}_P|^2|{\bm x}|^{2\gamma-2}dx+C^{\rm tor}_{N,\gamma}\int_{\mathbb{R}^N}|{\bm u}_T|^2|{\bm x}|^{2\gamma-2}dx
 +\mathcal{R}_{N,\gamma}[{\bm u}_T]
\\& 
 = C_{N,\gamma}\int_{\mathbb{R}^N}|{\bm u}|^2|{\bm x}|^{2\gamma-2}dx+\big(C^{\rm pol}_{N,\gamma}-C^{\rm tor}_{N,\gamma}\big)\int_{\mathbb{R}^N}|{\bm u}_P|^2|{\bm x}|^{2\gamma-2}dx
 +\mathcal{R}_{N,\gamma}[{\bm u}_T].
\end{split}
\]
By the density argument, the same inequality also applies to solenoidal fields ${\bm u}.$ 
If a solenoidal field ${\bm u}$ attains $C_{N,\gamma}$ in \eqref{H-L}, then this fact together with the condition $C^{\rm pol}_{N,\gamma}>C^{\rm tor}_{N,\gamma}$ (from Lemma~\ref{lemma:const1}) implies that 
\[
 \int_{\mathbb{R}^N}|{\bm u}_P|^2|{\bm x}|^{2\gamma-2}dx=0\quad\text{ and }\quad \mathcal{R}_{N,\gamma}[{\bm u}_T]=0,
\]
and hence that
\[
 {\bm u}_P\equiv {\bm 0}\quad\text{and}\quad {\bm x}\cdot\nabla \(|{\bm x}|^{\gamma+\frac{N-2}{2}}{\bm u}_T\)={\bm 0}.
\]
Here the second equation further implies that
\[{\bm u}_T({\bm x})=|{\bm x}|^{-\gamma-\frac{N-2}{2}}{\bm c}\(\frac{\bm{x}}{|\bm{x}|}\)\] 
for some vector field ${\bm c}$ on $\mathbb{S}^{N-1}$; this fact together with the integrability condition $\|{\bm u}_T\|_{\gamma}<\infty$ forces that ${\bm c}\equiv {\bm 0}$. 
Therefore, it turns out that ${\bm u}\equiv {\bm 0}$. 
\subsubsection*{The case $\gamma\not\in I_N$
}
The calculation of \eqref{HL_rem_P}$_{{\bm u}={\bm u}_P}$ plus \eqref{H-L_tor_rem}$_{{\bm u}={\bm u}_T}$ yields that
\[
\begin{aligned}
  \int_{\mathbb{R}^N}|\nabla {\bm u}|^2&|{\bm x}|^{2\gamma}dx
=  \int_{\mathbb{R}^N}|\nabla {\bm u}_P|^2|{\bm x}|^{2\gamma}dx+  \int_{\mathbb{R}^N}|\nabla {\bm u}_T|^2|{\bm x}|^{2\gamma}dx
\\& 
 \ge C^{\rm pol}_{N,\gamma}\int_{\mathbb{R}^N}|{\bm u}_P|^2|{\bm x}|^{2\gamma-2}dx+c_{N,\gamma}\mathcal{R}_{N,\gamma}[{\bm u}_P]
\\&\quad +C^{\rm tor}_{N,\gamma}\int_{\mathbb{R}^N}|{\bm u}_T|^2|{\bm x}|^{2\gamma-2}dx
 +\mathcal{R}_{N,\gamma}[{\bm u}_T]
\\& 
 \ge C_{N,\gamma}\int_{\mathbb{R}^N}|{\bm u}|^2|{\bm x}|^{2\gamma-2}dx
+c_{N,\gamma} \mathcal{R}_{N,\gamma}[{\bm u}]
\end{aligned}
\]
holds for all solenoidal fields ${\bm u}$ in $C_c^\infty(\dot{\mathbb{R}}^N)^N,$ and hence in $\mathcal{D}_\gamma(\mathbb{R}^N)$ by the density argument. 
On the other hand, in view of \eqref{c1}, we find $c_{N,\gamma}>0$ from $G(0)>0$; indeed, a direct computation yields from \eqref{c0} that 
\[
 \begin{split}
 {G}(0)=1-\mfrac{4(\gamma-1)(N-1)}{\(\big(\gamma-\frac{N}{2}\big)^2+N-1\)^2}=\mfrac{B\big(\gamma-\frac{N}{2}\big)}{\(\big(\gamma-\frac{N}{2}\big)^2+N-1\)^2}>0,
 \end{split}
\]
where the last inequality follows by applying the contraposition of Lemma~\ref{lemma:const2} to the inequality $C^{\rm tor}_{N,\gamma}\ge C^{\rm pol}_{N,\gamma}$ (which comes from Lemma~\ref{lemma:const1} and clearly contradicts to $C^{\rm tor}_{N,\gamma}+2\le C^{\rm pol}_{N,\gamma}$).    
 Therefore, 
 any solenoidal field ${\bm u}$ in $\mathcal{D}_\gamma(\mathbb{R}^N)$ which attains $C_{N,\gamma}$ in \eqref{H-L} must satisfy $\mathcal{R}_{N,\gamma}[{\bm u}]=0$, and hence we get ${\bm u}\equiv {\bm 0}$ in the same way as the previous case. 

In summary, the proof of the non-attainability of $C_{N,\gamma}$ 
is now complete, and we finish the proof of  Theorem~\ref{theorem:HL}.

\section*{\bf Acknowledgements}
This work is supported by JSPS KAKENHI Grant numbers JP22KJ2604 
and JP22K13943. The author thanks Prof. Y. Kabeya (Osaka Metropolitan University) for his great support and encouragement. Additionally, this research is partly supported by JSPS KAKENHI Grant-in-Aid for Scientific Research (B) 19H01800 (Prof. F. Takahashi) and Osaka Central Advanced Mathematical Institute (MEXTJoint Usage/Research Center on Mathematics and Theoretical Physics JPMXP0619217849).
\bibliographystyle{amsplain-it}
\bibliography{bibdata,bibmydata}

\end{document}